\documentclass[12pt]{amsart}
\usepackage{geometry}
\usepackage[plainpages,urlcolor=blue]{hyperref}
\usepackage{amscd}
\usepackage{amssymb}
\usepackage{mathrsfs}
\usepackage{amsmath}
\usepackage{amsthm}
\usepackage[all,cmtip]{xy}
\usepackage{enumerate}
\usepackage{enumitem}
\usepackage{ytableau}
\usepackage{tikz}
\usepackage{mathtools}
\usepackage{color}
\usepackage{wasysym}
\usepackage{comment}

\theoremstyle{plain}
\newtheorem{thm}{Theorem}[section]
\newtheorem{prop}[thm]{Proposition}
\theoremstyle{definition}
\newtheorem{defn}[thm]{Definition}
\theoremstyle{remark}
\newtheorem{rmk}[thm]{Remark}

\theoremstyle{lemma}
\newtheorem{lem}[thm]{Lemma}
\theoremstyle{Corollary}

\theoremstyle{example}
\newtheorem{exam}[thm]{Example}

\newcommand{\thmref}[1]{Theorem~\ref{#1}}
\newcommand{\secref}[1]{Section~\ref{#1}}
\newcommand{\lemref}[1]{Lemma~\ref{#1}}
\newcommand{\propref}[1]{Proposition~\ref{#1}}

\newcommand{\defref}[1]{Definition~\ref{#1}}

\newcommand{\exref}[1]{Example~\ref{#1}}

\newcommand{\C}{{\mathbb C}}
\newcommand{\R}{{\mathbb R}}
\newcommand{\Z}{{\mathbb Z}}
\newcommand{\N}{{\mathbb N}}

\newcommand{\CE}{{\mathcal E}}

\newcommand{\id}{{\rm{id}}}

\newcommand{\fU}{{\mathfrak U}}

\newcommand{\be}{\mathbf e}

\newcommand{\br}{\mathbf r}
\newcommand{\bt}{\mathbf t}
\newcommand{\bs}{\mathbf s}

\newcommand{\de}{\delta}

\newcommand{\U}{{\rm{U}}}

\newcommand{\Sym}{{\rm{Sym}}}

\newcommand{\im}{{\rm{Im}}}
\newcommand{\Ad}{{\rm{Ad}}}

\advance\headheight by 2pt

\newcommand{\ot}{\otimes}

\newcommand{\fsl}{{\mathfrak {sl}}}

\newcommand{\so}{{\mathfrak {so}}}

\newcommand{\gl}{{\mathfrak {gl}}}

\newcommand{\osp}{{\mathfrak {osp}}}

\newcommand{\Uq}{{\rm{U}}_q}

\newcommand{\Cl}{{\rm{Cl}}}
\newcommand{\Clq}{{\rm{Cl}}_q}

\newcommand{\vac}{|0\rangle}
\newcommand{\vacr}{|\br\rangle}
\newcommand{\vacs}{|\bs\rangle}
\newcommand{\vact}{|\bt\rangle}

\numberwithin{equation}{section}
\makeatletter
\def\moverlay{\mathpalette\mov@rlay}
\def\mov@rlay#1#2{\leavevmode\vtop{%
   \baselineskip\z@skip \lineskiplimit-\maxdimen
   \ialign{\hfil$\m@th#1##$\hfil\cr#2\crcr}}}
\newcommand{\charfusion}[3][\mathord]{
    #1{\ifx#1\mathop\vphantom{#2}\fi
        \mathpalette\mov@rlay{#2\cr#3}
      }
    \ifx#1\mathop\expandafter\displaylimits\fi}
\makeatother

\begin{document}
\title[Spinor Representation of $\Uq(\osp(1|2n))$]{A Quantum Fermion Realisation of  the Finite Dimensional Spinor Representation of $\Uq(\osp(1|2n))$}
\author[H Y Yang]{Hengyun Yang}
\address[H Y Yang]{ Department of Mathematics, Shanghai Maritime University, Shanghai 201306, China}
\email{hyyang@shmtu.edu.cn}

\author[Y Zhang]{Yang Zhang}
\address[Y  Zhang]{School of Mathematical Sciences, University of Science and Technology of China, Hefei 230026, China}
\address[H Y Yang, Y  Zhang]{School of Mathematics and Statistics, University of Sydney, Sydney 2006, Australia}
\email{yang91@mail.ustc.edu.cn}

\begin{abstract}
The quantum supergroup $\Uq(\osp(1|2n))$ admits a finite dimensional spinor representation, which does not have a classical limit. We construct a realisation of  this representation on the Fock space of $q$-fermions. We also generalise the construction to the infinite dimensional spinor representations of $\Uq(\osp(2m+1|2n))$ for $m, n\ge 1$ by using both quantum bosons and fermions. This leads to a new realisation different from those obtained before by other researchers. 
\end{abstract}
\date{Nov 29, 2016.}

\maketitle

\section{Introduction}
The Clifford algebra has a deformation  \cite{HS,H} that is compatible with the action of the quantum  group $\Uq(\so(m))$. This makes it possible to construct the quantum analogue of the spinor representation in terms of quantum fermions. The construction is similar to the classical case, where spinor representations of orthogonal Lie algebras arise from the corresponding Clifford algebras.  However, as far as we are aware, none has tried to generalise the construction to the quantum orthosymplectic supergroup \cite{BGZ, Y, ZBG91, ZGB91b} 
by using quantum fermions only.  What deterred people from doing so is the well known fact that the spinor representation  of the orthosymplectic Lie superalgebra $\osp(2m+1|2n)$ ($n>0$) is infinite dimensional. 

However, in the case of $\Uq(\osp(1|2n))$,  there exists a finite dimensional irreducible spinor 
representation, though this is not widely known. This representation does not have a classical (i.e., $q\to 1$) limit, thus its existence does not contravene the above mentioned fact.  

The aim of this note is to give a quantum fermion construction of this finite dimensional spinor representation of $\Uq(\osp(1|2n))$ (see \thmref{thm:alghomo} and \propref{PropSpinorSo}).  Our starting point is the observation that there exists an algebra homomorphism from $\Uq(\osp(1|2n))$ to $\U_{-q}(\so(2n+1))$ (cf. \cite{Z92}), which makes it possible to construct spinor representation of $\Uq(\osp(1|2n))$ using its counterpart of $\U_{-q}(\so(2n+1))$. This method fails in the general case of  $\Uq(\osp(2m+1|2n))$ for $m, n>0$ as there is no such algebra homomorphism between $\Uq(\osp(2m+1|2n))$ and any quantum orthogonal or symplectic groups; see \cite{XZ}. 

In order to generalise the construction to  $\Uq(\osp(2m+1|2n))$ for $m>0$, it is necessary to bring quantum bosons into the picture.  The quantum fermions and bosons together form a 
deformed Clifford superalgebra $\Cl_q(m|n)$, the structure of which has been studied extensively in the literature (see, e.g., \cite{DPS}).  However,  the deformed Clifford superalgebra is used  quite differently in this paper (see \secref{secQuanosp} and \defref{defnClsuper}).

Let $\Uq(\osp(2m+1|2n))$ be the quantum orthosymplectic group with Dynkin diagram shown in \eqref{eqDynDiag1} below.  One of our main results, shown in \thmref{thm-osp}, determines  an algebra homomorphism from  $\Uq(\osp(2m+1|2n))$ to deformed Clifford superalgebra $\Cl_q(m|n)$. This together with the Fock space of $\Cl_q(m,n)$ give an infinite dimensional irreducible representation of $\Uq(\osp(2m+1|2n))$ with the highest weight  
\[ (\underbrace{-1,\dots, -1}_{m-1},\underbrace{1,\dots,1}_{n},\sqrt{-1}q^{\frac{1}{2}}),
 \] 
which is referred to as spinor representation of $\Uq(\osp(2m+1|2n))$. In particular,  this reduces to the algebra homomorphism between $\Uq(\osp(1|2n))$ and deformed Clifford algebra $\Cl_q(n)$  when $m=0$, from which we obtain an  irreducible spinor representation of dimension $2^n$ for $\Uq(\osp(1|2n))$ with the  highest weight  	
\[ (\underbrace{1,1,\dots,1}_{n-1},\sqrt{-1}q^{\frac{1}{2}}),
\] 
see \thmref{thm:alghomo} and \propref{PropSpinorSo}. Of particular interest is that this spinor representation is not the deformation of any irreducible representation of $\osp(1|2n)$; see \exref{exosp12}. Comparing with  \cite[Theorem 3.1]{ZY}, we see that this is one of the integral representations of $\Uq(\osp(1|2n))$. 

 Note that  $\Uq(\gl_m)$ and  $\Uq(\osp(1|2n))$ are the regular subalgebras of $\Uq(\osp(2m+1|2n))$, and we have $\Uq(\osp(2m+1|2n))\supset \Uq(\gl_m)\otimes  \Uq(\osp(1|2n))$. In this paper, the  spinor representation of    $\Uq(\osp(2m+1|2n))$ is obtained by using  the  quantum boson realisation of  $\Uq(\gl_m)$ \cite{D} and  quantum  fermion realisation  of $\Uq(\osp(1|2n))$. However, our method can not be  applied to  the case of $\Uq(\osp(2m|2n))$,  since $\Uq(\osp(2m|2n))$ does not contain $\Uq(\osp(1|2n))$ as regular subalgebra.

\medskip
\noindent
{\bf{Relation to earlier work.}}
 Recall that there are various choices of  Borel subalgebras for the orthosymplectic Lie superalgebra $\osp(2m+1|2n)$  (see   \cite{K} and \cite{Z14}),  which are not Weyl group conjugate. They correspond to different simple root systems containing different sets of odd simple roots.   The  Dynkin diagrams of simple root systems of  $\osp(2m+1|2n)$ can be classified into the following two types:
\begin{center}
	\begin{tikzpicture}
	\node at (-2,0) {Type 1:};
	\node at (0,0) {\Large $\times$};
	\draw (0.2,0)--(1.0,0);
	\node at (1.2,0) {\Large $\times$};
	\draw (1.4,0)--(2.2,0);
	\node at (2.6,0) {\dots};
	\draw (3.0,0)--(3.8,0);	
	\node at (4.0,0) {\Large $\times$};
	\draw (4.2,0.05)--(5.0,0.05); 
	\draw (4.2,-0.05)--(5.0,-0.05);
	\node at (5.0,0) {\Large$>$};
	\draw [black,fill=black] (5.4,0) circle (0.2 cm);	
	
	\node at (0,-0.5) {\tiny $1$}; 
	\node at (1.2,-0.5) {\tiny $2$}; 
	\node at (4.0,-0.5) {\tiny $m+n-1$}; 
	\node at (5.4,-0.5) {\tiny $m+n$}; 
	\node at (6.0,-0.2){,};	
	\end{tikzpicture}
\end{center}

\begin{center}
	\begin{tikzpicture}
	\node at (-2,0) {Type 2:};
	\node at (0,0) {\Large $\times$};
	\draw (0.2,0)--(1.0,0);
	\node at (1.2,0) {\Large $\times$};
	\draw (1.4,0)--(2.2,0);
	\node at (2.6,0) {\dots};
	\draw (3.0,0)--(3.8,0);	
	\node at (4.0,0) {\Large $\times$};
	\draw (4.2,0.05)--(5.0,0.05); 
	\draw (4.2,-0.05)--(5.0,-0.05);
	\node at (5.0,0) {\Large$>$};
	\draw  (5.4,0) circle (0.2 cm);	
	
	\node at (0,-0.5) {\tiny $1$}; 
	\node at (1.2,-0.5) {\tiny $2$}; 
	\node at (4.0,-0.5) {\tiny $m+n-1$}; 
	\node at (5.4,-0.5) {\tiny $m+n$} ; 
	\node at (6.0,-0.2){,};	
	\end{tikzpicture}
\end{center}
where a node
$\circ$ corresponds to an even simple root;
$\otimes$ to an odd isotropic simple root;
$\bullet$ to an odd non-isotropic simple root, and
$\times$ stands for $\circ$ or $\otimes$, depending on whether the simple root is even or odd.
In particular, let  $\Theta_1=\{m,m+n\}$ be the index set for odd simple roots, then  the associated  Dynkin diagram  for $(\osp(2m+1|2n), \Theta_1 )$ is 
\begin{equation}\label{eqDynDiag1}
\begin{aligned}
  \begin{tikzpicture}
  \draw (0,0) circle (0.2 cm);
  \draw (0.2,0)--(0.8,0);
  \node at (1.2,0) {...};
  \draw (1.6,0)--(2.2,0);				
  \draw (2.4,0) circle (0.2 cm);
  \draw (2.6,0)--(3.4,0);       
  \node at (3.6,0) {\Large$\otimes$};
  \draw (3.8,0)--(4.6,0);
  \draw (4.8,0) circle (0.2 cm);
  \draw (5.0,0)--(5.6,0);
  \node at (6.0,0) {...};	
  \draw (6.4,0)--(7.0,0);	
  \draw (7.2,0) circle (0.2 cm);		
  \draw (7.4,0.05)--(8.2,0.05); 
  \draw (7.4,-0.05)--(8.2,-0.05);
  \node at (8.2,0) {\Large$>$};		 
  
  \draw [black,fill=black] (8.6,0) circle (0.2 cm);
  
  \node at (0,-0.5) {\tiny $1$}; 
  \node at (2.4,-0.5) {\tiny $m-1$}; 
  \node at (3.6,-0.5) {\tiny $m$}; 
  \node at (4.8,-0.5) {\tiny $m+1$}; 
  \node at (7.2,-0.5) {\tiny $m+n-1$}; 
  \node at (8.6,-0.5) {\tiny $m+n$}; 	
  \node at (9.3,-0.2) {.}; 		
  \end{tikzpicture}
\end{aligned}
\end{equation}
Another case,  denoted by  $(\osp(2m+1|2n),\Theta_2)$  with index set  $\Theta_2=\{n\}$ for odd simple roots, corresponds to the following  Dynkin diagram  
\begin{equation}\label{eqDynDiag2}
  \begin{aligned}
  \begin{tikzpicture}
  \draw (0,0) circle (0.2 cm);
  \draw (0.2,0)--(0.8,0);
  \node at (1.2,0) {...};
  \draw (1.6,0)--(2.2,0);				
  \draw (2.4,0) circle (0.2 cm);
  \draw (2.6,0)--(3.4,0);       
  \node at (3.6,0) {\Large$\otimes$};
  \draw (3.8,0)--(4.6,0);
  \draw (4.8,0) circle (0.2 cm);
  \draw (5.0,0)--(5.6,0);
  \node at (6.0,0) {...};	
  \draw (6.4,0)--(7.0,0);	
  \draw (7.2,0) circle (0.2 cm);		
  \draw (7.4,0.05)--(8.2,0.05); 
  \draw (7.4,-0.05)--(8.2,-0.05);
  \node at (8.2,0) {\Large$>$};		 
  
  \draw (8.6,0) circle (0.2 cm);
  \node at (0,-0.5) {\tiny $1$}; 	
  \node at (2.4,-0.5) {\tiny $n-1$};
  \node at (3.6,-0.5) {\tiny $n$};
  \node at (4.8,-0.5) {\tiny $n+1$};
  \node at (7.2,-0.5) {\tiny $m+n-1$}; 	
  \node at (8.6,-0.5) {\tiny $m+n$}; 
  \node at (9.3,-0.2) {.}; 	 	 			
  \end{tikzpicture}
  \end{aligned}
\end{equation}

 The quantum boson-fermion  realisation of  $\Uq(\osp(2m+1|2n),\Theta_2)$  and quantum fermion-boson  realisation of   $\Uq(\osp(2m+1|2n),\Theta_1)$ were studied in \cite{FSV},  while \cite{DPS}  constucted  the irreducible representations of  $\Uq(\osp(2m+1|2n),\Theta_2 )$. In this paper,  our definitions of $\Uq(\osp(2m+1|2n),\Theta_1)$ and $\Cl_q(m,n)$ are  different from the set up of \cite{DPS}, though our approach is similar to theirs. This leads to a new realisation of $\Uq(2m+1|2n,\Theta_1)$ (see \thmref{thm:alghomo} and \thmref{thm-osp}) different from those obtained before. In  \secref{secgen}, we shall generalise our method to obtain the realisation and spinor representation of $\Uq(2m+1|2n,\Theta)$ with arbitrary set $\Theta$ of odd simple roots; see \thmref{thmpigen} and \propref{propspingen}.
 
This paper is arranged as follows. In \secref{secosp12n},  we consider the quantum supergroup $\Uq(\osp(1|2n),\Theta_1)$ ($m=0$ case) associated with the Dynkin diagram \eqref{eqDynDiag1}. Using the deformed Clifford algebra $\Cl_{q}(n)$ and its finite dimensional Fock space $V_q(n)$, we construct the irreducible spinor representation of  $\Uq(\osp(1|2n),\Theta_1)$. In \secref{secospgen}, we generalise this method to the case $\Uq(\osp(2m+1|2n),\Theta_1)$, and obtain  the infinite dimensional irreducible spinor representation of $\Uq(\osp(2m+1|2n),\Theta_1)$, of which the underlying space is the Fock space $V_q(m,n)$ associated to the deformed Clifford algebra $\Clq(m,n)$. More generally, we give the  realisation of $\Uq(\osp(2m+1|2n),\Theta)$  for any set $\Theta$ of odd simple roots, and hence obtain the corresponding spinor representation of $\Uq(\osp(2m+1|2n),\Theta)$.

\section{Spinor Representation of $\Uq(\osp(1|2n),\Theta_1)$  }\label{secosp12n}

We shall construct a quantum fermion realisation of  the finite dimensional spinor  representation of $\Uq(\osp(1|2n),\Theta_1)$ in this section. We first recall the general  definition of quantum  orthosymplectic supergroup $\Uq(\osp(2m+1|2n),\Theta_1)$. Then we introduce the deformed  Clifford algebra  $\Clq(n)$ furnished with  the Fock space  $V_q(n)$. As well as  the classical case, we   present an algebra homomorphism from $\U_{q}(\osp(1|2n),\Theta_1)$ to $\Cl_{q}(n)$, which in turn gives us the  irreducible spinor representation $V_q(n)$ of $\U_{q}(\osp(1|2n),\Theta_1)$. 

\subsection{Quantum orthosymplectic supergroup $\Uq(\osp(2m+1|2n),\Theta_1)$}\label{secQuanosp}
Recall that $\Theta_1=\{m,m+n\}$ with corresponding Dynkin diagram shown in \eqref{eqDynDiag1}. 
Denote the  symmetrisable $(m+n)\times (m+n)$-Cartan matrix  $A=(a_{ij})$ of $\Uq(\osp(2m+1|2n),\Theta_1)$   by
\[
A=\begin{bmatrix}
A_{m-1} &   &  &   & &  &\\
&   & -1 &  &  &  &  \\
& -1 & 0 & 1 &   &   & \\
&   & -1 &  &   &  &  \\
&   &   &   &  A_{n-1} &  & \\
&   &   &  &  &  & -1 \\
&   &   &  & & -2& 2 \\
\end{bmatrix},
\]
where 
\[
A_n=\begin{bmatrix}
2 & -1  &  &   & \\
-1& 2  &  &  &    \\
& & \ddots  &  &   \\
&   &  &  2 & -1    \\
&   &   &  -1 & 2
\end{bmatrix}
\]stands for the $n\times n$ Cartan matrix of the general linear Lie algebra. 
Define the sequence 
\[(d_1\dots,d_m,d_{m+1},\dots, d_{m+n-1},d_{m+n})=(-1,\dots,-1,1,\dots,1,\frac{1}{2})  \]
such that  $d_{i}a_{ij}=d_ja_{ji}$. Let $q$ be an indeterminate over the complex field $\C$, and let $q_i=q^{d_i}$.

For any superalgebra $A=A_{\bar{0}}\oplus A_{\bar{1}}$, we denote by $[\;]: A\rightarrow \Z_2$ the parity functor, i.e., $[a]=\bar{0}$ if $a\in A_{\bar{0}}$ and  $[a]=\bar{1}$ if $a\in A_{\bar{1}}$.

\begin{defn}
	The quantum supergroup  $\Uq(\osp(2m+1|2n),\Theta_1)$ \cite{BGZ, Y, ZBG91, ZGB91b}  is an associative superalgebra over $\C(q^{\frac{1}{2}})$  generated by  $e_i,f_i,k_i^{\pm 1} (i=1,2,\dots, m+n)$ with the $\Z_2$-grading
	\[ 
	\begin{aligned}
	&[e_i]= [f_i]=\bar{0}
	, \quad i\notin \Theta_1,\quad [k_j]=\bar{0},\quad j=1,2, \dots, m+n,\\
	&[e_s]=[f_s]=\bar{1}, \quad s\in \Theta_1,
	\end{aligned}
	\]
	subject to the following relations
	\begin{align}
	k_ik_i^{-1}=k_i^{-1}k_i&=1,\quad k_ik_j=k_jk_i,\label{eq: quandef1}\\  
	k_ie_jk_i^{-1}=q_i^{a_{ij}}e_j&,\quad k_if_jk_i^{-1}=q_i^{-a_{ij}}f_j, \label{eq: quandef2}\\
	e_if_j-(-1)^{[e_i][f_j]}&f_je_i=\delta_{ij}\frac{k_i-k_i^{-1}}{q-q^{-1}}, \label{eq: quandef3}\\
	e_m^2&=f_m^2=0  ,  &  \label{eq: quandef4}\\
	(\Ad_{e_i})^{1-a_{ij}}(e_j)= (\Ad&_{f_i})^{1-a_{ij}}(f_j)=0\quad    \text{for} ~ a_{ii}\neq 0, i\neq j, \label{eq: quandef5}\\
	\Ad_{e_m}\Ad_{e_{m-1}}\Ad_{e_m}(e_{m+1})&=0,\quad    \Ad_{f_m}\Ad_{f_{m-1}}\Ad_{f_m}(f_{m+1})=0, \label{eq: quandef6}
	\end{align}
	where   $\Ad_{e_i}(x)$ and $\Ad_{f_i}(x)$ are defined by
	\begin{equation}
	\Ad_{e_i}(x)=e_ix-(-1)^{[e_i][x]}k_ixk_i^{-1}e_i,\quad \Ad_{f_i}(x)=f_ix-(-1)^{[f_i][x]}k_i^{-1}xk_if_i,
	\end{equation}
\end{defn}

Relations \eqref{eq: quandef4}-\eqref{eq: quandef5} and \eqref{eq: quandef6} are referred to as quantum Serre relations and higher order quantum Serre relations (see \cite{Y} or \cite{XZ}), respectively . 

\begin{rmk}
Note that the standard expression $\frac{k_i-k_i^{-1}}{q_i-q_i^{-1}}$ is replaced by $\frac{k_i-k_i^{-1}}{q-q^{-1}}$ in \eqref{eq: quandef3}.  As a result, $q^{\pm \frac{1}{2}}$ never appears in our definition. One can also define $\Uq(\osp(2m+1|2n),\Theta_1)$ over the field $\C(q)$, but we shall work over $\C(q^{\frac{1}{2}})$ in order to construct the spinor representations of $\Uq(\osp(2m+1|2n),\Theta_1)$.	
\end{rmk}

In the case of  $\Uq(\osp(1|2n),\Theta_1)$ (here $ \Theta_1=\{n\}$), the  Cartan matrix $A=(a_{ij})_{n\times n}$ is
\[A=\begin{bmatrix}
	A_{n-2}   &  & &  \\
	&     & 2& -1 \\
	&    & -2& 2 
\end{bmatrix},
\]
and  $(d_1,\dots,d_{n-1},d_n)=(1,\dots,1,\frac{1}{2})$ satisfies $d_ia_{ij}=d_ja_{ji}$. The associative superalgebra  $\Uq(\osp(1|2n),\Theta_1)$ is generated by  $e_i,f_i,k_i^{\pm 1}$ ( $i=1,2,\dots, n$) with the $\Z_2$-grading 
\[   
\begin{aligned}
 [e_i]=[f_i]=\bar{0}, \quad i\neq n,\quad [k_j]=\bar{0},\quad j=1,2,\dots, n, \quad
 [e_n]= [f_n]=\bar{1},
\end{aligned}
\]
satisfying relations (\ref{eq: quandef1})-(\ref{eq: quandef3}) and  (\ref{eq: quandef5}). The quantum Serre relations (\ref{eq: quandef5}) for generators $e_i$ ($i=1,2,\dots, n$)  can be written more explicitly as
\begin{align}
 e_ie_j-e_je_i=0,    &      \quad |i-j|\ge 2, \label{def-osp2n-1}\\
 e_{i\pm 1}^2e_i-(q+q^{-1}) e_{i\pm 1}&e_ie_{i\pm 1}+ e_ie_{i\pm 1}^2=0,\quad i\pm 1\neq n,\label{def-osp2n-2}\\
 e^3_ne_{n-1}+(1-q-q^{-1})(e^2_n&e_{n-1}e_n+e_ne_{n-1}e^2_n)+e_{n-1}e^3_n=0 \label{def-osp2n-3}.
\end{align}   
Similarly, replacing $e_i$ by $f_i$  in above equations yields quantum Serre relations for generators $f_i$  ($i=1,2,\dots, n$).

\subsection{Deformed  Clifford algebra  $\Clq(n)$}
Let $n$ be a positive integer. We define the deformed Clifford algebra $\Clq(n)$ over $\C(q^{\frac{1}{2}})$ with fermionic  generators $\psi_i,\psi_i^{\dagger}, v_i^{\pm 1} (i=1,2,\dots, n)$ by the following defining relations
\begin{align}
 v_iv_j=v_jv_i,&\quad  v_iv_i^{-1}=v_i^{-1}v_i=1,\label{def-Vq1}\\
 v_i\psi_jv_i^{-1}=q^{\delta_{ij}}&\psi_j,\quad v_i\psi_j^{\dagger}v_i^{-1}=q^{-\delta_{ij}}\psi_j^{\dagger},\label{def-Vq2}\\
 \psi_i\psi_j+\psi_j\psi_i&=\psi_i^{\dagger}\psi_j^{\dagger}+\psi_j^{\dagger}\psi_i^{\dagger}=0,  \label{def-Vq3}\\
 \psi_i\psi_j^{\dagger}+&\psi_j^{\dagger}\psi_i=0, \quad i\neq j, \label{def-Vq4}\\
 \psi_i\psi_i^{\dagger}+q\psi_i^{\dagger}\psi_i=&v_i^{-1},\quad \psi_i\psi_i^{\dagger}+q^{-1}\psi_i^{\dagger}\psi_i=v_i.\label{def-Vq5}
\end{align}
Note that relations  \eqref{def-Vq5} are equivalent to the following
\begin{equation} \label{eq:om1} \tag{\ref{def-Vq5}$'$}
 \psi_i\psi_i^{\dagger}=\frac{qv_i-(qv_i)^{-1}}{q-q^{-1}},\quad \psi_i^{\dagger}\psi_i=-\frac{v_i-v_i^{-1}}{q-q^{-1}}.
\end{equation}

\begin{lem}\label{lem-1}
For any $1\le i,j \le n-1$, the following equations hold in $\Clq(n)$
\begin{align}
    & [\psi_i\psi_{i+1}^{\dagger},\psi_{n}^{\dagger}]= [\psi_{i+1}\psi_i^{\dagger},\psi_{n}^{\dagger}]=0,\label{lem-1-12}\\
	& [\psi_i\psi_{i+1}^{\dagger},\psi_{j+1}\psi_j^{\dagger}]=\de_{i,j} \frac{v_iv_{i+1}^{-1}-(v_iv_{i+1}^{-1})^{-1}}{q-q^{-1}},\label{lem-1-1}\\
    & \psi_n\psi_n^{\dagger}+ \psi_n^{\dagger}\psi_n= \frac{q^{\frac{1}{2}}v_n+q^{-\frac{1}{2}}v_n^{-1}}{q^{\frac{1}{2}}+q^{-\frac{1}{2}}}. \label{lem-1-13}
\end{align}
\end{lem}
\begin{proof} Equations \eqref{lem-1-12} are obvious by (\ref{def-Vq3}). Equations  (\ref{lem-1-1}) follow easily from (\ref{def-Vq3}) and (\ref{def-Vq4}) whenever $i\neq j$. For $i=j$, using  (\ref{def-Vq2}) and ( \ref{eq:om1}), we obtain
\[  
  \begin{aligned}
     &\psi_i\psi_{i+1}^{\dagger}\psi_{i+1}\psi_i^{\dagger}-\psi_{i+1}\psi_i^{\dagger}\psi_i\psi_{i+1}^{\dagger}=\psi_i\psi_{i}^{\dagger}\psi_{i+1}^{\dagger}\psi_{i+1}-\psi_i^{\dagger}\psi_{i}\psi_{i+1}\psi_{i+1}^{\dagger}\\
    =&\frac{1}{(q-q^{-1})^2}((v_{i+1}^{-1}-v_{i+1})(qv_i-q^{-1}v_i^{-1})+(v_i-v_i^{-1})(qv_{i+1}-q^{-1}v_{i+1}^{-1}))\\
    =& \frac{v_iv_{i+1}^{-1}-v_i^{-1}v_{i+1}}{q-q^{-1}} .
  \end{aligned}   
\]
Thus  (\ref{lem-1-1}) holds for  $i=j$.
Finally, by (\ref{eq:om1}) we have 
$$\psi_n\psi_n^{\dagger}+ \psi_n^{\dagger}\psi_n=\frac{qv_n-(qv_n)^{-1}}{q-q^{-1}}-\frac{v_n-v_n^{-1}}{q-q^{-1}}= \frac{q^{\frac{1}{2}}v_n+q^{-\frac{1}{2}}v_n^{-1}}{q^{\frac{1}{2}}+q^{-\frac{1}{2}}}$$
as required in \eqref{lem-1-13}.
\end{proof}

 We now construct the Fock space $V_q(n)$ for $\Clq(n)$. The vacuum vector $\vac\in V_q(n)$ is defined in the usual way
\begin{equation*}
 \psi_i\vac=0,\quad v_i\vac=\vac,\quad i=1,2,\dots,n.
\end{equation*}
As a basis for $V_q(n)$, we set
\begin{equation*}
   \vacr:=(\psi_1^{\dagger})^{r_1}(\psi_2^{\dagger})^{r_2}\dots(\psi_n^{\dagger})^{r_n}\vac,
\end{equation*}
where $\br=(r_1,r_2,\dots,r_n)$ with $r_i\in\{0,1\}$.  Note that $V_q(n) $ is $\Z_{+}$-graded with the degree assignments $\deg \psi^{\dagger}_{i}=1$ for all $1\leq i\leq n$ (here $\Z_{+}$ denotes the set of non-negative integers). We set $\deg \vacr=\sum_{i=1}^{n}r_i$  and order the basis elements lexicographically, i.e.,
\begin{equation}\label{eqlex}
\vacr \succ \vacs, \quad \text{if} \  r_i>s_i\  \text{for the first}\ i \ \text{where}\ \ r_i\ \text{and }\  s_i\ \text{differ}.  
\end{equation}
For notational convenience, we denote by $\be_i$ the $n$-tuple with 1 in the $i$-th position and 0 elsewhere.

\begin{prop}\label{prop:irrcl}
	The Fock space $V_q(n)$  is an irreducible $\Clq(n)$ module of dimension $2^n$ under the action
	\begin{equation}\label{Fock-1}
	\begin{aligned}
	v_i\vacr&=q^{-r_i}\vacr,\\
	\psi_i\vacr&=(-1)^{r_1+\cdots+r_{i-1}}|\br-\be_i\rangle,\\
	\psi_i^{\dagger}\vacr&=(-1)^{r_1+\cdots+r_{i-1}}|\br+\be_i\rangle,\quad i=1,2,\dots, n,
	\end{aligned}	
	\end{equation}
	where $|\br\pm\be_{i}\rangle=0$ if $r_i\pm 1\notin \{0,1\}$.
\end{prop}
\begin{proof}
	The action \eqref{Fock-1} follows immediately from the defining relations for $\Clq(n)$, so it remains to prove the irreducibility of $V_q(n)$.  Assume that  $W_q\subseteq V_q(n)$ is a nonzero submodule and  $w=\sum_{\br}a_{\br}\vacr (a_{\br}\in\C)$ is a nonzero element in $W_q$ with the leading term $a_{\bs}\vacs$ such that $a_{\bs}\neq 0$ and  $\vacs$ is maximal with respect to the lexicographic ordering.
	Then $(\psi_1)^{s_1}\cdots(\psi_n)^{s_n}w=\pm a_{\bs}\vac\in W_q$, which is nonzero and hence $W_q=V_q(n)$. 
\end{proof}



\subsection{Spinor representation of  $\U_{q}(\osp(1|2n),\Theta_1)$}
We proceed with the construction of   spinor representation of  $\U_{q}(\osp(1|2n),\Theta_1)$. 

\begin{thm}\label{thm:alghomo}
	The $\C(q^{\frac{1}{2}})$-linear map $\pi: \Uq(\osp(1|2n),\Theta_1) \rightarrow  \Cl_{q}(n)$ assigned by
	\[\begin{aligned}
	&e_i \mapsto \psi_{i}\psi_{i+1}^{\dagger},\quad  f_i \mapsto \psi_{i+1}\psi_i^{\dagger}\quad k_i\mapsto v_i v_{i+1}^{-1},\quad  i=1,2, \dots, n-1, \\
	& e_n \mapsto \sqrt{-1}(q^{\frac{1}{2}}-q^{-\frac{1}{2}})^{-1}\psi_n,\quad f_n\mapsto \psi_n^{\dagger},\quad k_n\mapsto \sqrt{-1} q^{\frac{1}{2}}v_n,
	\end{aligned}\]	
  is an associative superalgebra homomorphism, where $\sqrt{-1}$ is the imaginary unit.
\end{thm}
\begin{proof}
	We first show that the assignments of generators preserve the relations (\ref{eq: quandef1})-(\ref{eq: quandef3})  for $1\leq i,j\leq n$.
	Relations (\ref{eq: quandef1}) directly follow from (\ref{def-Vq1}). Relations (\ref{eq: quandef2}) are obvious whenever $|i-j|\ge 2$, while in the case of $j=i+1$ we have for the first relations of \eqref{eq: quandef2} that 
	\begin{align*}
	&\pi(k_i)\pi(e_{i+1})\pi(k_i^{-1})=v_iv_{i+1}^{-1}\psi_{i+1}\psi_{i+2}^{\dagger}v_{i+1}v_i^{-1}=q^{-1}\pi(e_{i+1}), \quad i=1,2,\dots, n-2,\\
	&\pi(k_i)\pi(e_{i})\pi(k_i^{-1})=v_iv_{i+1}^{-1}\psi_{i}\psi_{i+1}^{\dagger}v_{i+1}v_i^{-1}=q^{2}\psi_{i}\psi_{i+1}^{\dagger}=q^{2}\pi(e_{i}),\quad i=1,2,\dots, n-1,\\
	&\pi(k_{n-1})\pi(e_{n})\pi(k_{n-1}^{-1})=\sqrt{-1}(q^{\frac{1}{2}}-q^{-\frac{1}{2}})^{-1}v_{n-1}v_{n}^{-1}\psi_{n}v_nv_{n-1}^{-1}=q^{-1}\pi(e_{n}).	
	\end{align*}
The other remaining relations of $\eqref{eq: quandef2}$ can be proved in the same way. By \lemref{lem-1},  relations (\ref{eq: quandef3})  hold under $\pi$ for all $1\le i,j \le n$.

 Now we turn to show that  $\pi$ preserves quantum  Serre relations (\ref{def-osp2n-1})-(\ref{def-osp2n-3}), and the case of quantum Serre relations for $f_i$ can be proved along the same line.  Relations (\ref{def-osp2n-1}) are obvious  from    (\ref{lem-1-1}).  Using (\ref{def-Vq3}), (\ref{def-Vq4}) and (\ref{def-Vq5}), we obtain 
\[  
  \begin{aligned}
 \pi(e_{i-1}e_i-qe_ie_{i-1})=&\psi_{i-1}\psi_{i}^{\dagger}\psi_{i}\psi_{i+1}^{\dagger}-q \psi_{i}\psi_{i+1}^{\dagger}\psi_{i-1}\psi_{i}^{\dagger}\\
=& (\psi_{i}^{\dagger}\psi_{i}-q \psi_{i}\psi_{i}^{\dagger})\psi_{i-1}\psi_{i+1}^{\dagger}\\
=& v_i^{-1}\psi_{i-1}\psi_{i+1}^{\dagger}.
\end{aligned}   
\]
This  together with (\ref{def-Vq2}) and (\ref{def-Vq3}) lead to
\[  
  \begin{aligned}
& \pi( (e_{i- 1})^2e_i-(q+q^{-1}) e_{i- 1}e_ie_{i-1}+ e_i(e_{i- 1})^2)\\
=& \pi(e_{i-1}) \pi(e_{i-1}e_i-qe_ie_{i-1})-q^{-1} \pi(e_{i-1}e_i-qe_ie_{i-1})\pi(e_{i-1}) \\
=& \psi_{i-1}\psi_i^{\dagger} v_i^{-1}\psi_{i-1}\psi_{i+1}^{\dagger}-q^{-1}v_i^{-1}\psi_{i-1}\psi_{i+1}^{\dagger}\psi_{i-1}\psi_i^{\dagger}\\
=& 0.
\end{aligned}   
\]	
Thus (\ref{def-osp2n-2}) holds for the case $i-1$, while  the case for $i+1$ can be proved similarly.  Finally, relation (\ref{def-osp2n-3}) follows immediately from $\psi_n^2=0$.
\end{proof}

Recall that the Fock space $V_q(n)$ is an irreducible representation of $\Cl_q(n)$. Therefore,  we  get an action of  $\U_{q}(\osp(1|2n),\Theta_1)$ on $V_q(n)$  through the above homomorphism $\pi$. This is called  \emph{ spinor representation} of $\U_{q}(\osp(1|2n),\Theta_1)$.

\begin{prop}\label{PropSpinorSo}
	The spinor representation  $V_q(n)$ of  $\U_{q}(\osp(1|2n),\Theta_1)$  is irreducible with the  highest weight vector $\vac$ of weight  $(1,\dots,1,\sqrt{-1}q^{\frac{1}{2}})$.  Furthermore, the dimension of  $V_q(n)$ is  $2^n$.
\end{prop}
\begin{proof}
	It remains to prove that $V_q(n) $ is irreducible as  $\U_{q}(\osp(1|2n),\Theta_1)$-module. Assume that $U\subset V_q(n)$ is a nonzero $\U_{q}(\osp(1|2n),\Theta_1)$ submodule. Let $f=\sum_{\br}a_{\br}\vacr$ be a nonzero element in $U$ such that all terms in $f$ are ordered lexicographically and the nonzero leading term $a_{\bs} \vacs$ is the maximal one. Setting $\deg f=\deg \vacs=\sum_{i=1}^{n}s_i$ and using induction on $\deg f$, we aim to show that there exists $N\in\Z_{+}$ and nonzero $c_q\in \C(q)$ such that $e_{i_1}e_{i_2}\dots e_{i_{N}}f=c_q\vac$ ($1\leq i_1,\dots,i_N\leq n$), and hence the irreducibility follows.
	
	There is nothing to prove when $\deg f=0$  (in this case, $N=0$), which means that $f$ is a scalar multiple of $\vac$. Now we turn to the general case that $\deg f=k\in \Z_{+}$ with the assumption that   it is true for  $\deg f=k-1$.   In the leading term
	$ a_{\bs}\vacs=a_{\bs} (\psi_1^{\dagger})^{s_1}(\psi_2^{\dagger})^{s_2}\dots(\psi_n^{\dagger})^{s_n}\vac,$ we have $s_{i}\in\{0,1\}$ for all $1\leq i\leq n$. Let $j$ ($1\leq j\leq n$) be the maximal index such that $s_j=1$, then we obtain
	\[  e_ne_{n-1}\dots e_{j} \vacs= \pm \sqrt{-1}(q^{\frac{1}{2}}-q^{-\frac{1}{2}})^{-1}|\bs-\be_{j}\rangle, \]
    which is still maximal among all terms of $e_ne_{n-1}\dots e_{j}f$ with respect to the lexicographical ordering. This reduces to the case $\deg(e_ne_{n-1}\dots e_{j}f)=\deg (e_ne_{n-1}\dots e_{j}\vacs)=k-1$, which is true by assumption.	
\end{proof}


\begin{exam}\label{exosp12}
	{\rm 
	The quantum supergroup 	$\Uq(\osp(1|2),\Theta_1)$ is generated by $e,f,k^{\pm 1}$ subject to relations 
	\[ kek^{-1}=qe,\quad kfk^{-1}=q^{-1}f,\quad ef+fe=\frac{k-k^{-1}}{q-q^{-1}}.   \]
	The spinor representation $V_q(1)$ is spanned by basis  $\{\vac,|1\rangle \}$ with the action given by
	\begin{equation*}
	\begin{aligned}
	&k\vac=\sqrt{-1}q^{\frac{1}{2}}\vac, \quad k|1 \rangle =-\sqrt{-1}q^{-\frac{1}{2}}|1 \rangle,\\
	& e\vac=0,\quad e|1\rangle=\sqrt{-1}(q^{\frac{1}{2}}-q^{-\frac{1}{2}})^{-1}\vac,\\
	&f\vac=|1\rangle, \quad f |1\rangle=0.
	\end{aligned}
    \end{equation*}
    It is clear that $\vac$ is the highest weight vector of weight $\sqrt{-1}q^{\frac{1}{2}}$. As a byproduct,  $V_q(1)$ is not a deformation of any finite dimensional representation of $\osp(1|2)$, since the dimension of irreducible representation of $\osp(1|2)$ must be odd, while the dimension of $V_q(1)$ is even. 
	 }	
\end{exam}

\section{Spinor representations of $\Uq(\osp(2m+1|2n))$}\label{secospgen}

In this section,  we  introduce the deformed Clifford superalgebra $\Cl_q(m,n)$ furnished with infinite dimensional irreducible Fock space, which is shown to be spinor representation of $\Uq(\osp(2m+1|2n),\Theta_1)$ with $\Theta_1=\{m,m+n\}$.  Then we generalise our construction to the case with arbitrary set $\Theta$ of odd simple roots and obtain the spinor  representation of $\Uq(\osp(2m+1|2n),\Theta)$.

\subsection{Simple case with $\Theta_1$} \label{secsimcase}
Let us start by introducing a set of  bosonic  generators $\phi_i, \phi_i^{\dagger},u_i^{\pm 1}$ ($i=1,2,\dots, m$), which are subject to the relations
\begin{align}
 u_iu_j=u_ju_i,&\quad  u_iu_i^{-1}=u_i^{-1}u_i=1,\label{def-Vq1'}\\
 u_i\phi_ju_i^{-1}=q^{-\delta_{ij}}&\phi_j,\quad u_i\phi_j^{\dagger}u_i^{-1}=q^{\delta_{ij}}\phi_j^{\dagger},\label{def-Vq2'}\\
 \phi_i\phi_j-\phi_j\phi_i&=\phi_i^{\dagger}\phi_j^{\dagger}-\phi_j^{\dagger}\phi_i^{\dagger}=0,  \label{def-Vq3'}\\
 \phi_i\phi_j^{\dagger}-&\phi_j^{\dagger}\phi_i=0, \quad i\neq j, \label{def-Vq4'}\\
 \phi_i\phi_i^{\dagger}-q\phi_i^{\dagger}\phi_i=&u_i^{-1},\quad \phi_i\phi_i^{\dagger}-q^{-1}\phi_i^{\dagger}\phi_i=u_i.\label{def-Vq5'}
\end{align}
Note that the relations  \eqref{def-Vq5'} are equivalent to the following
\begin{equation} \label{eq:om1'} \tag{\ref{def-Vq5'}$'$}
 \phi_i\phi_i^{\dagger}=\frac{qu_i-(qu_i)^{-1}}{q-q^{-1}},\quad \phi_i^{\dagger}\phi_i=\frac{u_i-u_i^{-1}}{q-q^{-1}}.
\end{equation}

\begin{defn}\label{defnClsuper}
	The \emph{deformed Clifford superalgebra $\Clq(m,n)$} is  the associative superalgebra over $\C(q^{\frac{1}{2}})$ generated by 
	bosonic  generators  $\phi_i, \phi_i^{\dagger},u_i^{\pm 1}$ ($i=1,2,\dots, m$) and fermionic  generators $\psi_j, \psi_j^{\dagger},v_j^{\pm 1}$ ($j=1,2,\dots, n$) with the $\Z_2$-grading
    \begin{equation}\label{eqgrading}
      \begin{aligned}
     	& [u_i^{\pm 1}]= [v_j^{\pm 1}]=  [\phi_i]=  [\phi_i^{\dagger}]=\bar{0}, \\
     	&[\psi_j]=[\psi_j^{\dagger}]=\bar{1},
      \end{aligned}	    
    \end{equation}
 subject to the relations (\ref{def-Vq1})-(\ref{def-Vq5}), (\ref{def-Vq1'})-(\ref{def-Vq5'}), and the condition that the bosonic and fermionic generators commute.
\end{defn}

We are now in a position to construct Fock space $V_q(m,n)$ of $\Clq(m,n)$. Let $\vac\in V_q(m,n)$  be the vacuum vector, which satisfies 
\begin{equation*}
  \phi_i\vac=\psi_j\vac=0,\quad u_i\vac=v_j\vac=\vac,\quad  i=1,2,\dots,m,  j=1,2,\dots,n.
\end{equation*}
Denote by 
\begin{equation*}
   \vact:=(\phi_1^{\dagger})^{t_1}(\phi_2^{\dagger})^{t_2}\dots(\phi_m^{\dagger})^{t_m}(\psi_1^{\dagger})^{\bar{t}_1}(\psi_2^{\dagger})^{\bar{t}_2}\dots(\psi_n^{\dagger})^{\bar{t}_n}\vac
\end{equation*}
the basis element for $V_q(m,n)$,  where $\bt=(t_1,t_2,\dots, t_m,\bar{t}_1,\bar{t}_2,\dots,\bar{t}_n)$ with $t_i\in\Z_+ $ and $ \bar{t}_j\in\{0,1\}$ for all  $1\leq i\leq m$ and   $1\leq j\leq n$.
Let $[t_i]_q=(q^{t_i}-q^{-t_i})(q-q^{-1})^{-1}$ be the $q$-integer and  $\be_i$ be the $(m+n)$-tuple with 1 in the $i$-th position and 0 elsewhere, the action of $\Clq(m,n)$ on  $V_q(m,n)$ is determined by
\begin{equation*}
   \begin{aligned}
   & u_i\vact=q^{t_i}\vact,\quad 
    \phi_i\vact=[t_i]_q|\bt-\be_i\rangle,\quad 
    \phi_i^{\dagger}\vact=|\bt+\be_i\rangle,  \quad i=1,2,\dots, m,\\
     &v_j\vact=q^{-\bar{t}_j}\vact,\quad \psi_j\vact=(-1)^{\bar{t}_1+\bar{t}_2+\cdots+\bar{t}_{j-1}}|\bt-\be_{m+j}\rangle,\\ &\psi_j^{\dagger}\vact=(-1)^{\bar{t}_1+\bar{t}_2+\cdots+\bar{t}_{j-1}}|\bt+\be_{m+j}\rangle,\quad j=1,2,\dots, n.
   \end{aligned}	
\end{equation*}
Note that we set $|\bt\pm\be_{m+j}\rangle=0$  whenever $\bar{t}_j\pm 1\notin \{0,1\}$.
By using similar method as in \propref{prop:irrcl}, we conclude that $V_q(m,n)$ is an infinite dimensional irreducible representation of $\Clq(m,n)$.

\begin{thm}\label{thm-osp}
The $\C(q^{\frac{1}{2}})$-linear map $\pi$ from $ \Uq(\osp(2m+1|2n),\Theta_1)$ to  $\Clq(m,n)$, defined by 
\[\begin{aligned}
	& e_i \mapsto \phi_{i}\phi_{i+1}^{\dagger},\quad  f_i \mapsto \phi_{i+1}\phi_i^{\dagger}\quad k_i\mapsto -u_i u_{i+1}^{-1},\quad  i=1, 2,\dots, m-1, \\
	& e_m \mapsto \phi_m\psi_1^{\dagger},\quad f_m\mapsto \phi_m^{\dagger}\psi_1,\quad k_m\mapsto u_mv_1^{-1},\\ 
          & e_{m+j} \mapsto \psi_{j}\psi_{j+1}^{\dagger},\quad  f_{m+j} \mapsto \psi_{j+1}\psi_j^{\dagger}\quad k_{m+j}\mapsto v_j v_{j+1}^{-1},\quad  j=1,2, \dots, n-1, \\
	& e_{m+n} \mapsto \sqrt{-1}(q^{\frac{1}{2}}-q^{-\frac{1}{2}})^{-1}\psi_n,\quad f_{m+n}\mapsto \psi_n^{\dagger},\quad k_{m+n}\mapsto \sqrt{-1} q^{\frac{1}{2}}v_n
\end{aligned}
\]
is an associative superalgebra homomorphism.
\end{thm}
\begin{proof}
 We need to show that  the assignments of generators preserve the relations (\ref{eq: quandef1})-(\ref{eq: quandef6}). The  proof of   (\ref{eq: quandef1})-(\ref{eq: quandef5}) is similar to that of  Theorem \ref{thm:alghomo}, so we just prove that $\pi$ preserves the first relation of (\ref{eq: quandef6}), which can be written explicitly as
\begin{equation}  \label{ospmn-1}
  \begin{aligned}
& e_me_{m-1}(e_me_{m+1}-q^{-1}e_{m+1}e_m)-qe_m(e_me_{m+1}-q^{-1}e_{m+1}e_m)e_{m-1}\\
& +e_{m-1}(e_me_{m+1}-q^{-1}e_{m+1}e_m)e_m-q(e_me_{m+1}-q^{-1}e_{m+1}e_m)e_{m-1}e_m=0.
 \end{aligned}
\end{equation}
By using Definition \ref{defnClsuper}, we get 
 \[
\pi(e_me_{m+1}-q^{-1}e_{m+1}e_m)=   \phi_m\psi_{2}^{\dagger}(\psi_1^{\dagger}\psi_{1}+q^{-1}\psi_{1}\psi_1^{\dagger})\\
=q^{-1} \phi_m\psi_{2}^{\dagger}v_1^{-1}.
\]
Hence the image of the left hand side of  (\ref{ospmn-1}) equals
\[
\begin{aligned}
& q^{-1} \phi_m\psi_1^{\dagger} \phi_{m-1}\phi_{m}^{\dagger}\phi_m\psi_{2}^{\dagger}v_1^{-1}-\phi_m\psi_1^{\dagger}\phi_m\psi_{2}^{\dagger}v_1^{-1}\phi_{m-1}\phi_{m}^{\dagger}\\
&+ q^{-1}\phi_{m-1}\phi_{m}^{\dagger}\phi_m\psi_{2}^{\dagger}v_1^{-1} \phi_m\psi_1^{\dagger}- \phi_m\psi_{2}^{\dagger}v_1^{-1} \phi_{m-1}\phi_{m}^{\dagger}                        \phi_m\psi_1^{\dagger}\\
=& ( \phi_m^2\phi_{m}^{\dagger} + \phi_{m}^{\dagger}\phi_m^2-(q+q^{-1})\phi_m\phi_{m}^{\dagger}\phi_m)\phi_{m-1}\psi_{2}^{\dagger}\psi_{1}^{\dagger}v_1^{-1},\\
\end{aligned}
\]
which is zero since
 \begin{equation}\label{e-relation}
  \begin{aligned}
&\phi_m^2\phi_{m}^{\dagger} + \phi_{m}^{\dagger}\phi_m^2-(q+q^{-1})\phi_m\phi_{m}^{\dagger}\phi_m\\
=&\phi_m u_m^{-1}+q u_m^{-1}\phi_m+(1+q^2) \phi_{m}^{\dagger}\phi_m^2-(q+q^{-1})( u_m^{-1}\phi_m+q\phi_{m}^{\dagger}\phi_m^2  )\\
=&0.
\end{aligned}
\end{equation}
\end{proof}

  Note that $V_q(m,n)$ is $\Z_{+}$-graded with the degree assignments $\deg \phi^{\dagger}_i=\deg \psi^{\dagger}_{j}=1$ for all $1\leq i\leq m$ and $1\leq j\leq n$.  Let $\bs,\bt \in  \Z_{+}^{\times m}\times \{ 0,1\}^{\times n}$ be the $(m+n)$-tuples,  we introduce the lexicographical-ordering  $ \vacs \succ \vact $ if there exists integer $a$ $(1\leq a\leq m+n)$ such that the $a$-th entry of $\bs$ is strictly bigger than the $a$-th entry of $\bt$, while  the first $a-1$ entries of $\bs$ and $\br$ are identical. For any  $f=\sum_{\bt}a_{\bt}\vact$ with the nonzero leading term $a_{\bs}\vacs$ such that $\vacs$ is maximal in the lexicographical ordering, we define its degree by setting
\[\deg f=\deg \vacs=\sum_{i=1}^{m}s_i+\sum_{j=1}^{n}\bar{s}_{j}. \]
Then we have the spinor representation $V_q(m,n)$ of $\U_{q}(\osp(2m+1|2n),\Theta_1)$ through the above homomorphism $\pi$.

\begin{prop}\label{PropSpinor-osp}  
	The spinor representation  $V_q(m,n)$ of  $\U_{q}(\osp(2m+1|2n),\Theta_1)$  is irreducible with the highest weight vector $\vac$ of weight  $(-1,\dots, -1,1,\dots,1,\sqrt{-1}q^{\frac{1}{2}})$, where there are $m-1$ copies of $-1$ and $n$ copies of $1$.
\end{prop}
\begin{proof}
	It can be proved by using the same trick as in \propref{PropSpinorSo}. Taking a nonzero element $f=\sum_{\bt}a_{\bt}\vact\in V_q(m,n)$, we  use induction on $\deg f$ to show the assertion that there exist $N\in\Z_+$ and nonzero $c_{q}\in \C(q)$ such that 
	$e_{i_1}e_{i_2}\dots e_{i_{N}}f=c_q \vac$, 
	where $1\leq i_1,i_2,\dots, i_N\leq m+n$.
	
    There is nothing to prove when $\deg f=0$, and we turn to the general case that $\deg f=k\in \Z_{+}$ with the assumption that   it is true for  $\deg f\leq k-1$.  Suppose $a_{\bs}\vacs$ is the leading term of $f$, we have
	\[ 	
	\deg (e_{m-1}^{s_1+s_2+\dots +s_{m-1}}\dots e_2^{s_1+s_2}e_1^{s_1}f)=\deg (e_{m-1}^{s_1+s_2+\dots +s_{m-1}}\dots e_2^{s_1+s_2}e_1^{s_1}\vacs)=\deg (c_q^{\prime} |\bs^{\prime}\rangle ),    \]
	where $0\neq c_q^{\prime}\in \C(q)$,  $\bs^{\prime}=(0,\dots,0,\sum_{i=1}^{m}s_i,\bar{s}_1,\dots,\bar{s}_n)$ and 
	\[ |\bs^{\prime}\rangle=(\phi^{\dagger}_m)^{ \sum_{i=1}^{m}s_i }(\psi_1^{\dagger})^{\bar{s}_1}(\psi_2^{\dagger})^{\bar{s}_2}\dots(\psi_n^{\dagger})^{\bar{s}_n}\vac.   \]
    Note that $\bar{s}_i \in \{0,1\}$ for all $1\leq i\leq n$, we have two cases now. If $\sum_{j=1}^{n}\bar{s}_j\neq 0$, then by the assertion proved in  \propref{PropSpinorSo} there exist $1\leq j_1,j_2,\dots, j_{N'}\leq n$ ($N^{\prime}\geq 1$) and $0\neq c_q^{\prime\prime}\in \C(q)$ such that 
	\[
	\begin{aligned}
	& \deg ( e_{m+j_1}e_{m+j_2}\dots e_{m+j_{N^{\prime}}} e_{m-1}^{s_1+s_2+\dots +s_{m-1}}\dots e_2^{s_1+s_2}e_1^{s_1}f)\\
	=& \deg (e_{m+j_1}e_{m+j_2}\dots e_{m+j_{N^{\prime}}} |\bs^{\prime}\rangle) =\deg (c_q^{\prime\prime} (\phi^{\dagger}_m)^{ \sum_{i=1}^{m}s_i }\vac )<k,
	\end{aligned}
	\]
	whence our assertion follows by assumption. Otherwise, we have $\bar{s}_1=\bar{s}_2=\dots=\bar{s}_n=0$. In this case $|\bs^{\prime}\rangle =(\phi^{\dagger}_m)^{ \sum_{i=1}^{m}s_i }\vac$, and 
	\[ \deg (e_m |\bs^{\prime}\rangle )=\deg ( (\phi^{\dagger}_m)^{ \sum_{i=1}^{m}s_i-1}\psi_1^{\dagger} \vac)=k,   \]
	which reduces to the first case as $(\phi^{\dagger}_m)^{ \sum_{i=1}^{m}s_i-1}\psi_1^{\dagger} \vac$ is the form considered therein.
\end{proof}


\subsection{General case with arbitrary $\Theta$}\label{secgen}
\subsubsection{General definition of $\Uq(\osp(2m+1|2n))$}
We begin with the root data of Lie superalgebra $\osp(2m+1|2n)$.  Let $\CE_{m|n}$ be the  $(m+n)$-dimensional vector space over $\R$ with a basis consisting  of elements $\varepsilon_i$ ($i=1, 2, \dots, m$) and $\delta_\mu$ ($\mu=1, 2, \dots, n$).  We endow the basis elements with a total order $\lhd$, which is called \emph{admissible order} if $\epsilon_i\lhd \epsilon_{i+1}$ and $\delta_{\mu}\lhd \delta_{\mu+1}$ for all $i,\mu$. Fix an admissible order and let $\CE_1\lhd \CE_2\lhd \cdots\lhd \CE_{m+n}$ be the ordered basis of $\CE_{m|n}$. We  define a symmetric non-degenerate bilinear form on $\CE_{m|n}$ by
\begin{equation}\label{eqbilinear}
(\epsilon_i,\epsilon_j)=-\delta_{ij},\quad (\delta_\mu,\delta_\nu)=\delta_{\mu\nu},\quad (\epsilon_i,\delta_\mu)=(\delta_\mu,\epsilon_i)=0.
\end{equation}

The set of positive roots of $\osp(2m+1|2n)$ can be  realised as a subset of $\CE_{m|n}$ with an admissible order $\lhd$. Explicitly, each choice of a Borel subalgebra of $\osp(2m+1|2n)$ corresponds to a choice of positive roots, and hence a fundamental system $\Pi^{\lhd}=\{\alpha_1,\alpha_2,\dots, \alpha_{m+n}\}$ of simple roots, where 
\[ \alpha_i=\CE_i-\CE_{i+1}, \quad  i=1,2, \dots, n-1,\quad \alpha_{m+n}=\CE_{m+n}.  \]
The Weyl group conjugacy classes of  Borel subalgebras correspond bijectively to the admissible ordered bases of $\CE_{m|n}$.
The Dynkin diagram associated to  $\Pi^{\lhd}$ is of Type 1 or Type 2 shown in Section 1, depending on whether $\alpha_{m+n}=\CE_{m+n}=\delta_n$ or $\alpha_{m+n}=\CE_{m+n}=\epsilon_m$.

Denote by $\Theta\subset\{1, 2, \dots, m+n\}$  the labelling set of the odd simple roots in $\Pi^{\lhd}$, that is, $\alpha_s\in \Pi^{\lhd}$ is an odd simple root for any $s\in\Theta$.  Clearly, the simple root system $\Pi^{\lhd}$ uniquely determines the set of odd roots, and vice versa.
We define the Cartan matrix of $\osp(2m+1|2n)$ associated to the simple root system $\Pi^{\lhd}$ (or equivalently, $\Theta$)  by
\begin{equation}\label{eqCartan}
  A=(a_{ij}) \quad\text{with}\quad
  a_{ij}=\begin{cases}\dfrac{2(\alpha_i,\alpha_j)}{(\alpha_i,\alpha_i)},&\mbox{if}~(\alpha_i,\alpha_i)\neq 0,\\-(\alpha_i,\alpha_j),&\mbox{if}~(\alpha_i,\alpha_i)=0.
  \end{cases}
\end{equation}
Note that $(\alpha_i,\alpha_i)=0$ if and only if $\alpha_i$ is an isotropic odd root. 
Let 
\[
d_i=\begin{cases}\dfrac{(\alpha_i,\alpha_i)}{2},&\mbox{if}~(\alpha_i,\alpha_i)\neq 0,\\    -1, &\mbox{if}~(\alpha_i,\alpha_i)=0,
\end{cases}
\]
 we have $d_{i}a_{ij}=d_ja_{ji}$. Let $q_i=q^{d_i}$. To introduce general definition of $\Uq(\osp(2m+1|2n),\Theta)$, we adopt the following notation:
\[\begin{aligned}
&[k]_{q_i}=\frac{{q_i}^k-{q_i}^{-k}}{{q_i}-{q_i}^{-1}},\quad \{k\}_{q_i}=\frac{{q_i}^k-(-{q_i})^{-k}}{{q_i}+{q_i}^{-1}},\quad \mbox{for}\ \  k\in\N, \\
&[0]_{q_i}!=\{0\}_{q_i}!=1, \quad [N]_{q_i}!=\prod_{i=1}^N[i]_{q_i},\quad \{N\}_{q_i}!=\prod_{i=1}^N\{i\}_{q_i}, \mbox{ for}\ \ 1\le N\in\N,\\
&\begin{bmatrix} N\\k\end{bmatrix}_{q_i}=\frac{[N]_{q_i}!}{[N-k]_{q_i}![k]_{q_i}!},\quad \left\{\begin{matrix} n\\k\end{matrix}\right\}_{q_i}=\frac{\{N\}_{q_i}!}{\{N-k\}_{q_i}! \{k\}_{q_i}!},\quad \mbox{for}\ \  k\leq N \in \N.
\end{aligned}
\]

In what follows, we write $\Uq(\osp(2m+1|2n),\Pi^{\lhd}):=\Uq(\osp(2m+1|2n),\Theta)$ to emphasise the significance of the admissible order $\lhd$.  

\begin{defn}\label{defospgen} \cite{XZ,Y}
The quantum supergroup $\Uq(\osp(2m+1|2n),\Pi^{\lhd})$  associated to $\Pi^{\lhd}$ over $\C(q^{\frac{1}{2}})$ is generated by $e_i,f_i,k_i^{\pm 1} (i=1,2,\dots, m+n)$, where $e_s,f_s (s\in\Theta)$ are odd and the rest are even. The defining relations are:
 \begin{enumerate}
 \item [(R1)] relations \eqref{eq: quandef1}, \eqref{eq: quandef2} and \eqref{eq: quandef3};
 \item [(R2)] Serre relations: 
 	\noindent if $i\notin \Theta$, then for all $j\ne i$,
 	\[
 	\begin{aligned}
 	&\sum_{k=0}^{1-a_{i j}}
 	(-1)^k \begin{bmatrix} 1-a_{i j}\\k\end{bmatrix}_{q_i} e_i^{k} e_j
 	e_i^{1-a_{i j}-k}=0,\quad \\
 	&\sum_{k=0}^{1-a_{i j}}
 	(-1)^k \begin{bmatrix} 1-a_{i j}\\k\end{bmatrix}_{q_i} f_i^{k} f_j
 	f_i^{1-a_{i j}-k}=0;
 	\end{aligned}
 	\]
 	\noindent  
 	if $s\in\Theta$ and $a_{s s}=2$, then for all $j\ne s$,
 	\begin{eqnarray*}
 	&&\label{eq:UqB-odd2}
 	\begin{aligned}
 	&\sum_{k=0}^{1-a_{s j}}
 	(-1)^{\frac{1}{2}k(k+1)} \left\{\begin{array}{c}1-a_{s j}\\k\end{array}\right\}_{q_s} e_s^{k} e_j
 	e_s^{1-a_{s j}-k}=0, \\
 	&\sum_{k=0}^{1-a_{s j}}
 	(-1)^{\frac{1}{2}k(k+1)}
 	\left\{\begin{array}{c} 1-a_{s j}\\k\end{array}\right\}_{q_s} f_s^{k} f_j
 	f_s^{1-a_{s j}-k}=0;
 	\end{aligned}
 	\end{eqnarray*}
 	if $a_{s s}=0$, then
 	\begin{eqnarray*}
 	&&\label{eq:UqB-odd0}
 	(e_s)^2=0, \quad (f_s)^2=0;
 	\end{eqnarray*}	
  \item [(R3)] high order Serre relations: if the Dynkin diagram
 	contains sub-diagrams of the following types:
 	{\renewcommand\baselinestretch{1.15}\selectfont
 		\begin{enumerate}
 			\item
 			\begin{picture}(75, 15)(0, 7)
 			\put(10, 7){$\times$}
 			\put(15, 10){\line(1, 0){20}}
 			\put(35, 6){\Large$\otimes$ }
 			\put(47, 10){\line(1, 0){20}}
 			\put(62, 7){$\times$}
 			\put(8, -2){\tiny $s-1$}
 			\put(39, -2){\tiny $s$}
 			\put(60, -2){\tiny $s+1$}
 			\put(70, 7){,}
 			\end{picture}
 			with $a_{s-1,s}=-a_{s,s+1}$, the associated 
 			\medskip
 			\noindent
 			higher order Serre relations are
 			\begin{eqnarray*}\label{eq:UqB-higher order1}
 			\begin{aligned}
 			&e_s e_{s-1; s; s+1}+(-1)^{[e_{s-1}]+[e_{s+1}]}  e_{s-1; s; s+1}e_s=0,\\
 			&f_s f_{s-1; s; s+1} +(-1)^{[f_{s-1}]+[f_{s+1}]} f_{s-1; s; s+1}f_s=0;
 			\end{aligned}
 			\end{eqnarray*}
 			
 			\item
 			\begin{picture}(80, 20)(0, 7)
 			\put(10, 7){$\times$}
 			\put(15, 10){\line(1, 0){20}}
 			\put(35, 6){\Large$\otimes$}
 			\put(47, 11){\line(1, 0){17}}
 			\put(47, 9){\line(1, 0){17}}
 			\put(57, 6.5){$>$}
 			\put(70,10){\circle{10}}
 			\put(8, -2){\tiny $s-1$}
 			\put(39, -2){\tiny $s$}
 			\put(63, -2){\tiny $s+1$}
 			\put(76, 7){,}
 			\end{picture}
 			where $s=m+n-1$ and $\alpha_{m+n}=\varepsilon_m$, and the associated
 			
 			\medskip
 			\noindent  higher order Serre relations are
 			\begin{eqnarray*}\label{eq:UqB-higher order2}
 			\begin{aligned}
 			&e_s e_{s-1; s; s+1}+(-1)^{[e_{s-1}]}  e_{s-1; s; s+1}e_s=0,\\
 			&f_s f_{s-1; s; s+1} +(-1)^{[f_{s-1}]}  f_{s-1; s; s+1}f_s=0;
 			\end{aligned}
 			\end{eqnarray*}
 			
 			\item
 			\begin{picture}(80, 20)(0, 7)
 			\put(10, 6.5){$\times$}
 			\put(15, 10){\line(1, 0){20}}
 			\put(35, 6){\Large$\otimes$ }
 			\put(47, 11){\line(1, 0){17}}
 			\put(47, 9){\line(1, 0){17}}
 			\put(57, 6.5){$>$}
 			\put(70, 10){\circle*{10}}
 			\put(8, -2){\tiny $s-1$}
 			\put(39, -2){\tiny $s$}
 			\put(63, -2){\tiny $s+1$}
 			\put(76, 7){,}
 			\end{picture}
 			where $s=m+n-1$ and $\alpha_{m+n}=\delta_n$, and the associated
 				
 				\medskip
 				\noindent  higher order Serre relations are
 			\begin{eqnarray*}\label{eq:UqB-higher order3}
 			\begin{aligned}
 			&e_s e_{s-1; s; s+1}-(-1)^{[e_{s-1}]}  e_{s-1; s; s+1}e_s=0,\\
 			&f_s f_{s-1; s; s+1} -(-1)^{[f_{s-1}]}  f_{s-1; s; s+1}f_s=0;
 			\end{aligned}
 			\end{eqnarray*}
 		\end{enumerate}}
 		where
 		\begin{eqnarray*}\label{eq:Uab-higher order ef}
 		\begin{aligned}
 		e_{i;s;j}=&e_i(e_se_j-(-1)^{[e_j]}q_j^{a_{js}}e_je_s)\\
 		&-(-1)^{[e_i](1+[e_j])}q_i^{a_{is}+a_{ij}}(e_se_j-(-1)^{[e_j]}q_j^{a_{js}}e_je_s)e_i,\\
 		f_{i;s;j}=&f_i(f_sf_j-(-1)^{[f_j]}q_j^{a_{js}}f_jf_s)\\
 		&-(-1)^{[f_i](1+[f_j])}q_i^{a_{is}+a_{ij}}(f_sf_j-(-1)^{ [f_j]}q_j^{a_{js}}f_jf_s)f_i.
 		\end{aligned}
 		\end{eqnarray*}		
 \end{enumerate} 
\end{defn}

We write $\epsilon_{m+\mu}:=\delta_\mu$ for all  $1\leq \mu\leq n$. For later use, we  give a precise characterisation of the set of all admissible orders on $\epsilon_{i}$ ($1\leq i\leq m+n$). Let $\Sym_{m+n}$ be the symmetric group of degree $m+n$ on the set $\{1,2,\dots, m+n\}$.  Recall that an $(m,n)$-shuffle in $\Sym_{m+n}$ is an permutation $\sigma\in \Sym_{m+n}$ such that $\sigma(i)<\sigma(j)$ for all $1\leq i<j\leq m$ or $m<i<j\leq m+n$. Let $\epsilon_1\epsilon_2\cdots\epsilon_{m+n}$ be a word. Then the $(m,n)$-shuffle $\sigma$ acts on it by permuting the subscripts, i.e., 
\begin{equation}\label{eqshuact}
  \sigma.\epsilon_1\epsilon_2\cdots\epsilon_{m+n}:= \epsilon_{\sigma^{-1}(1)}\epsilon_{\sigma^{-1}(2)}\cdots\epsilon_{\sigma^{-1}(m+n)}. 
\end{equation}
Clearly, $\CE_{i}=\epsilon_{\sigma^{-1}(i)}$ ($i=1,2,\dots, m+n$) form an ordered basis for $\CE_{m|n}$. This means that any $(m,n)$-shuffle determines uniquely an admissible order on $\epsilon_{i}$, and vice versa. Thus, we arrive at the following lemma.
\begin{lem}
	The set of admissible orders $\lhd$ on $\epsilon_{i}, i=1,2,\dots, m+n$,  corresponds bijectively to the set of $(m,n)$-shuffles $\sigma$ in $\Sym_{m+n}$.
\end{lem}

\subsubsection{Realisation of $\Uq(\osp(2m+1|2n))$}
Fix an admissible order $\lhd$ on $\epsilon_{i},1\leq i\leq m+n$, and let $\sigma\in\Sym_{m+n}$ be the corresponding $(m,n)$-shuffle.
Now we proceed to give a general definition of $\Clq(m,n)^{\lhd}$ arising from the  $\lhd$. For notational convenience,  we denote
\[ \phi_{m+j}:=\psi_{j},\quad \phi_{m+j}^{\dagger}:=\psi_{j}^{\dagger},\quad u_{m+j}^{\pm 1}:=v_j^{\pm 1},\quad  j=1,2,\dots, n.  \]
Then the $(m,n)$-shuffle $\sigma$ acts on the word $\phi_1\phi_2\cdots\phi_{m+n}$ as in \eqref{eqshuact}, and we introduce $\Phi_{i}:=\phi_{\sigma^{-1}(i)}$ for all $1\leq i\leq m+n$. Similarly, we introduce the elements $\Phi_{i}^{\dagger}:=\phi_{\sigma^{-1}(i)}^{\dagger}$ and $\fU_{i}^{\pm 1}:=u_{\sigma^{-1}(i)}^{\pm 1}$ for all $1\leq i\leq m+n$. Now the gradings \eqref{eqgrading} naturally induce the $\Z_2$-gradings on $\Phi_{i}, \Phi_{i}^{\dagger}$ and $\fU_{i}$, that is,
\[ [\Phi_{i}]=[\phi_{\sigma^{-1}(i)}],\quad [\Phi_{i}^{\dagger}]=[\phi_{\sigma^{-1}(i)}^{\dagger}],\quad [\fU_{i}^{\pm 1}]=\bar{0},\quad 1\leq i\leq m+n.    \]
\begin{defn}\label{defclgen}
	The deformed Clifford superalgebra $\Clq(m,n)^{\lhd}$ over  $\C(q^{\frac{1}{2}})$ is generated by $\Phi_{i}, \Phi_{i}^{\dagger}$ and $\fU_{i}$ $ (i=1,2,\dots, m+n)$, subject to the following relations:
	 \begin{align*}
	 \fU_i\fU_j=\fU_j\fU_i,&\quad \fU_i \fU_i^{-1}=\fU_i^{-1}\fU_i=1,\\
	 \fU_i\Phi_j\fU_i^{-1}=q^{-(-1)^{[\Phi_j]}\delta_{ij}}&\Phi_j,\quad \fU_i\Phi_j^{\dagger}\fU_i^{-1}=q^{(-1)^{[\Phi_j]}\delta_{ij}}\Phi_j^{\dagger},\\
	 \Phi_i\Phi_j-(-1)^{[\Phi_i][\Phi_j]}\Phi_j\Phi_i&=\Phi_i^{\dagger}\Phi_j^{\dagger}-(-1)^{[\Phi_i^{\dagger}][\Phi_j^{\dagger}]}\Phi_j^{\dagger}\Phi_i^{\dagger}=0,  \\
	 \Phi_i\Phi_j^{\dagger}-(-1)&^{[\Phi_i][\Phi_j^{\dagger}]}\Phi_j^{\dagger}\Phi_i=0, \quad i\neq j, \\
	 \Phi_i\Phi_i^{\dagger}-(-1)^{[\Phi_i]}q\Phi_i^{\dagger}\Phi_i=&\fU_i^{-1},\quad \Phi_i\Phi_i^{\dagger}-(-1)^{[\Phi_i]}q^{-1}\Phi_i^{\dagger}\Phi_i=\fU_i.
	 \end{align*}
\end{defn} 
Note that \defref{defnClsuper} is the special case of \defref{defclgen} with $(m,n)$-shuffle $\sigma=\id$.  The last two equations are equivalent to
\[  \Phi_i\Phi_i^{\dagger}=\frac{q\fU_i-(q\fU_i)^{-1}}{q-q^{-1}},\quad \Phi_i^{\dagger}\Phi_i=(-1)^{[\Phi_i]}\frac{\fU_i-\fU_i^{-1}}{q-q^{-1}}.    \]

\begin{thm}\label{thmpigen}
		The $\C(q^{\frac{1}{2}})$-linear map $\pi^{\lhd}$ from $ \Uq(\osp(2m+1|2n),\Pi^{\lhd})$ to  $\Clq(m,n)^{\lhd}$, defined by 
		\[\begin{aligned}
		& e_i \mapsto \Phi_{i}\Phi_{i+1}^{\dagger},\quad  f_i \mapsto \Phi_{i+1}\Phi_i^{\dagger}\quad k_i\mapsto -(-1)^{[\Phi_{i+1}]}\fU_i \fU_{i+1}^{-1},\quad  i=1, 2,\dots, m+n-1, \\
		& e_{m+n} \mapsto \sqrt{-1}(q^{\frac{1}{2}}-q^{-\frac{1}{2}})^{-1}\Phi_{m+n},\quad f_{m+n}\mapsto \Phi_{m+n}^{\dagger},\quad k_{m+n}\mapsto \sqrt{-1} q^{\frac{1}{2}}\fU_{m+n},
		\end{aligned}
		\]
		is an associative superalgebra homomorphism.
\end{thm}
\begin{proof}
  For the Dynkin diagram of Type 1, the proof of Theorem \ref{thm-osp} has already shown that the map $\pi^{\lhd}$ preserves the defining relations (R1) and (R2) in \defref{defospgen}. Similarly, we can prove  as in Theorem \ref{thm-osp} that  $\pi^{\lhd}$ preserves the defining relations (R1) and (R2) in \defref{defospgen}  for the Dynkin diagram of Type 2.
  
  Now we need to show that $\pi^{\lhd}$ also respects the higher order Serre relations (R3) related to diagrams (a), (b) and (c). The cases with all $\times$'s corresponding to even simple roots in diagrams (a) and (c) have also been covered by the proof of Theorem \ref{thm-osp}. The remaining cases with some or all $\times$'s being isotropic odd simple roots $\ot$ can all be proven in a similar way. Thus we will illustrate the proof by considering as an example the following case of diagram (a):
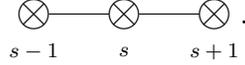
	\begin{center}
		\begin{tikzpicture}
		\node at (0,0) {\Large $\otimes$};
		\draw (0.2,0)--(1.0,0);
		\node at (1.2,0) {\Large $\otimes$};
		\draw (1.4,0)--(2.2,0);
		\node at (2.4,0) {\Large $\otimes$};
		
		\node at (0,-0.5) {\tiny $s-1$}; 
		\node at (1.2,-0.5) {\tiny $s$};
		\node at (2.4,-0.5) {\tiny $s+1$};  
		\node at (2.8, -0.1) {.};	
		\end{tikzpicture}
	\end{center}
There are two cases for the simple odd roots associated to this sub-diagram.
	
Case 1: the associated simple roots are
\[ \alpha_{s-1}=\epsilon_{s-1}-\delta_{s-1},\quad \alpha_{s}=\delta_{s-1}-\epsilon_{s}, \quad \alpha_{s+1}=\epsilon_s-\delta_{s}. \] 
Then we have 
\[
\begin{aligned}
&\Phi_{s-1}=\phi_{s-1},\quad \Phi_{s}=\psi_{s-1},\quad \Phi_{s+1}=\phi_{s},\quad\Phi_{s+2}=\psi_s,\\
&\Phi_{s-1}^{\dagger}=\phi_{s-1}^{\dagger},\quad \Phi_{s}^{\dagger}=\psi_{s-1}^{\dagger},\quad \Phi_{s+1}^{\dagger}=\phi_{s}^{\dagger},\quad\Phi_{s+2}^{\dagger}=\psi_s^{\dagger},\\
&\fU_{s-1}^{\pm 1}=u_{s-1}^{\pm 1},\quad \fU_{s}^{\pm 1}=v_{s-1}^{\pm 1},\quad \fU_{s+1}^{\pm 1}=u_{s}^{\pm 1},\quad \fU_{s+1}^{\pm 1}=v_{s}^{\pm 1}.
\end{aligned}
 \]
Thus, the map $\pi^{\lhd}$ on corresponding generators  is given by
\begin{equation}\label{case 1}\begin{aligned}
		& e_{s-1} \mapsto \phi_{s-1}\psi_{s-1}^{\dagger},\quad  f_{s-1} \mapsto \psi_{s-1}\phi_{s-1}^{\dagger},\quad k_{s-1}\mapsto u_{s-1} v_{s-1}^{-1}, \\
		& e_{s} \mapsto \psi_{s-1}\phi_{s}^{\dagger},\quad  f_{s} \mapsto \phi_{s}\psi_{s-1}^{\dagger},\quad k_s\mapsto -v_{s-1} u_{s}^{-1}, \\
		& e_{s+1} \mapsto \phi_{s}\psi_{s}^{\dagger},\quad  f_{s+1} \mapsto \psi_{s}\phi_{s}^{\dagger},\quad k_{s+1}\mapsto u_{s} v_{s}^{-1}.
	\end{aligned}
\end{equation}
We shall prove that the images of (\ref{case 1}) preserve the first relation given in diagram (a), the second relation can be proved along the same line. Recalling the bilinear form in \eqref{eqbilinear} and Cartan matrix \eqref{eqCartan}, we have 
$$a_{s-1,s}=1,\quad a_{s-1,s+1}=0,\quad a_{s+1,s}=-1,\quad q_{s-1}=q_{s+1}=q^{-1}.$$
and  $e_{s-1;s;s+1}=e_{s-1}(e_se_{s+1}+qe_{s+1}e_s)-q^{-1}(e_se_{s+1}+qe_{s+1}e_s)e_{s-1}$.
Note that 
$$\pi^{\lhd}(e_se_{s+1}+qe_{s+1}e_s)=\psi_{s-1}\psi_{s}^{\dagger}(\phi_{s}^{\dagger}\phi_{s}-q\phi_{s}\phi_{s}^{\dagger})=-q\psi_{s-1}\psi_{s}^{\dagger}u_{s}.$$
Therefore,
$$\pi^{\lhd}(e_{s-1;s;s+1})=\psi_{s-1}\psi_{s}^{\dagger}u_{s}\phi_{s-1}\psi_{s-1}^{\dagger}-q\phi_{s-1}\psi_{s-1}^{\dagger}\psi_{s-1}\psi_{s}^{\dagger}u_{s}.
$$
By using Definition \ref{defclgen}, we obtain 
\[
\begin{aligned}
&\pi^{\lhd}(e_s e_{s-1; s; s+1}+(-1)^{[e_{s-1}]+[e_{s+1}]}  e_{s-1; s; s+1}e_s)\\
=& \psi_{s-1}\phi_{s}^{\dagger}(\psi_{s-1}\psi_{s}^{\dagger}u_{s}\phi_{s-1}\psi_{s-1}^{\dagger}-q\phi_{s-1}\psi_{s-1}^{\dagger}\psi_{s-1}\psi_{s}^{\dagger}u_{s})\\
&+(\psi_{s-1}\psi_{s}^{\dagger}u_{s}\phi_{s-1}\psi_{s-1}^{\dagger}-q\phi_{s-1}\psi_{s-1}^{\dagger}\psi_{s-1}\psi_{s}^{\dagger}u_{s})\psi_{s-1}\phi_{s}^{\dagger}\\
=&\phi_{s-1}\phi_{s}^{\dagger}(q^2\psi_{s-1}^{\dagger}\psi_{s-1}^2-\psi_{s-1}^2\psi_{s-1}^{\dagger})\psi_{s}^{\dagger}u_{s}\\
=&0.
\end{aligned}
\]

Case 2: the associated simple roots are
\[ \alpha_{s-1}=\delta_{s-1}-\epsilon_{s-1},\quad \alpha_{s}=\epsilon_{s-1}-\delta_{s}, \quad \alpha_{s+1}=\delta_s-\epsilon_{s}. \]
Similarly, we have the map $\pi^{\lhd}$ on corresponding generators given by
\begin{equation}\label{case 2}\begin{aligned}
		& e_{s-1} \mapsto \psi_{s-1}\phi_{s-1}^{\dagger},\quad  f_{s-1} \mapsto \phi_{s-1}\psi_{s-1}^{\dagger},\quad k_{s-1}\mapsto -v_{s-1}u_{s-1}^{-1}, \\
		& e_{s} \mapsto \phi_{s-1}\psi_{s}^{\dagger},\quad  f_{s} \mapsto \psi_{s}\phi_{s-1}^{\dagger},\quad k_s\mapsto u_{s-1}v_{s}^{-1}, \\
		& e_{s+1} \mapsto \psi_{s}\phi_{s}^{\dagger},\quad  f_{s+1} \mapsto \phi_{s}\psi_{s}^{\dagger},\quad k_{s+1}\mapsto -v_{s}u_{s}^{-1}.
	\end{aligned}
\end{equation}
In this case, 
$$e_{s-1;s;s+1}=e_{s-1}(e_se_{s+1}+q^{-1}e_{s+1}e_s)-q(e_se_{s+1}+q^{-1}e_{s+1}e_s)e_{s-1},$$
$$\pi^{\lhd}(e_se_{s+1}+q^{-1}e_{s+1}e_s)=\phi_{s-1}\phi_{s}^{\dagger}(\psi_{s}^{\dagger}\psi_{s}+q^{-1}\psi_{s}\psi_{s}^{\dagger})=q^{-1}\phi_{s-1}\phi_{s}^{\dagger}v_{s}^{-1}.$$
By using Definition \ref{defclgen} and relation (\ref{e-relation}), we obtain 
\[
\begin{aligned}
&\pi^{\lhd}(e_s e_{s-1; s; s+1}+(-1)^{[e_{s-1}]+[e_{s+1}]}  e_{s-1; s; s+1}e_s)\\
=& \phi_{s-1}\psi_{s}^{\dagger}(q^{-1}\psi_{s-1}\phi_{s-1}^{\dagger}\phi_{s-1}\phi_{s}^{\dagger}v_{s}^{-1}-\phi_{s-1}\phi_{s}^{\dagger}v_{s}^{-1}\psi_{s-1}\phi_{s-1}^{\dagger})\\
&+(q^{-1}\psi_{s-1}\phi_{s-1}^{\dagger}\phi_{s-1}\phi_{s}^{\dagger}v_{s}^{-1}-\phi_{s-1}\phi_{s}^{\dagger}v_{s}^{-1}\psi_{s-1}\phi_{s-1}^{\dagger})\phi_{s-1}\psi_{s}^{\dagger}\\
=&\psi_{s-1}\psi_{s}^{\dagger}((q+q^{-1})\phi_{s-1}\phi_{s-1}^{\dagger}\phi_{s-1}-\phi_{s-1}^2\phi_{s-1}^{\dagger}-\phi_{s-1}^{\dagger}\phi_{s-1}^2)\phi_{s}^{\dagger}v_{s}^{-1}\\
=&0.
\end{aligned}
\]
The second relation given in diagram (a) can be proved in the same way. This completes the proof of our example,
\end{proof}

We now construct the  Fock space $V_q(m,n)^{\lhd}$ of $\Clq(m,n)^{\lhd}$.
As a vector space, $\Clq(m,n)^{\lhd}$ is spanned by the basis elements
\[ \vact:=\prod_{i=1}^{m+n}(\Phi_{i}^{\dagger})^{t_i}\vac, \quad \bt=(t_1,t_2,\dots,t_{m+n}),\]
where $t_i\in \Z_{+}$ if $[\Phi_{i}^{\dagger}]=\bar{0}$, and $t_i\in \Z_{2}$ if $[\Phi_{i}^{\dagger}]=\bar{1}$, and $\vac$ is the vacuum vector such that
\[ \Phi_i\vac=0, \quad \fU_i\vac=\vac, \quad i=1,2,\dots, m+n.  \] 
The action of $\Clq(m,n)^{\lhd}$ on the Fock space  $V_q(m,n)^{\lhd}$ can be derived explicitly from the defining relations in \defref{defclgen}.
By using similar method as in \propref{prop:irrcl}, we conclude that $V_q(m,n)^{\lhd}$ is an infinite dimensional irreducible representation of $\Clq(m,n)^{\lhd}$. 


\begin{prop}\label{propspingen}
	The  spinor representation  $V_q(m,n)^{\lhd}$ of  $\U_{q}(\osp(2m+1|2n),\Pi^{\lhd})$  is irreducible with the highest weight vector $\vac$ of weight  
	\[ (-(-1)^{[\Phi_2]},-(-1)^{[\Phi_3]},\dots, -(-1)^{[\Phi_{m+n}]},\sqrt{-1}q^{\frac{1}{2}}).   \]
\end{prop}
\begin{proof}
	It suffices to show that the superalgebra homomorphism $\pi^{\lhd}$ given in \thmref{thmpigen} is surjective, since  $V_q(m,n)^{\lhd}$ is irreducible  as $\Clq(m,n)^{\lhd}$-module. 
	
	It is easy to see that the image $\im\, \pi^{\lhd}$ contains generators $\fU_{i}$, $i=1,2,\dots, m+n$. Since
	\[ 
	\begin{aligned}
	   &(\Phi_{m+n-1}^{\dagger}\Phi_{m+n})\Phi^{\dagger}_{m+n}-(-1)^{[\Phi_{m+n}]+[\Phi^{\dagger}_{m+n-1}][\Phi^{\dagger}_{m+n}]} q^{-1}\Phi^{\dagger}_{m+n}(\Phi^{\dagger}_{m+n-1}\Phi_{m+n})\\
	   =& \Phi_{m+n-1}^{\dagger} (\Phi_{m+n}\Phi^{\dagger}_{m+n}-(-1)^{[\Phi_{m+n}]}\Phi^{\dagger}_{m+n}\Phi_{m+n} )\\
	   =&  \Phi_{m+n-1}^{\dagger}\fU_{m+n}.
	\end{aligned}	
	\]
   and  $\fU_{m+n}^{\pm 1}, \Phi^{\dagger}_{m+n}, \Phi_{m+n-1}^{\dagger}\Phi_{m+n}\in \im\, \pi^{\lhd}$,  we obtain that  $\Phi_{m+n-1}^{\dagger} \in \im \, \pi^{\lhd}$. Continuing similar calculations, we conclude that  $\im\, \pi^{\lhd}$ contains all generators $\Phi_{i}^{\dagger}$, $i=1,2,\dots, m+n$. It can be proved in a similar way that the generators  $\Phi_{i}$, $i=1,2,\dots, m+n$ are contained in $\im\, \pi^{\lhd}$. Therefore, $\pi^{\lhd}$ is surjective.
\end{proof}
\begin{rmk}
  	Though \propref{PropSpinorSo} and \propref{PropSpinor-osp} can be deduced immediately from \propref{propspingen}, their proofs are interesting in their own rights.
\end{rmk}

\begin{rmk}
	We remark that there is an associative algebra isomorphism between the quantum supergroup $\Uq(\osp(1|2n),\Theta_1)$ and quantum group $\U_{-q}(\so(2n+1))$; see \cite{Z92} and \cite[Theorem 1.1]{XZ}. Any  representation, especially spinor representation (cf. \cite{H}), of $\U_{-q}(\so(2n+1))$  can be translated to $\Uq(\osp(1|2n))$ via the preceding isomorphism. In particular, we note that the highest weight of spinor representation of $\U_{-q}(\so(2n+1))$ is $(1,\dots,1,(-q)^{\frac{1}{2}})$, which is essentially the same as that in \propref{PropSpinorSo}. More generally, the above mentioned isomorphism does exist for any fundamental system of $\osp(1|2n)$ (see \cite[Theorem 1.1]{XZ}), which ensures the existence of spinor representation for  $\Uq(\osp(1|2n),\Theta)$ with arbitrary $\Theta$. Although there is no such associative algebra homomorphism from $\Uq(\osp(2m+1|2n),\Theta)$ to any quantum group, we are able to generalise this construction in  \thmref{thmpigen} and \propref{propspingen} by means of deformed Clifford superalgebra $\Clq(m,n)$.
%
%
\end{rmk}

\noindent
{\bf Acknowledgement.}  This work was done during the authors' visit at the University of Sydney. The authors thank Professor Ruibin Zhang for discussions and help in preparing this paper.  This work is  supported by China Scholarship Council  and  Natural Science Foundation of Shanghai (Grant No. 16ZR1415000).

\end{document}